\documentclass[11pt,a4paper, reqno]{amsart}
\usepackage[colorlinks,citecolor=blue]{hyperref}
\usepackage{amsmath}
\usepackage{amsthm}
\usepackage{amssymb}
\usepackage{enumerate}
\usepackage{color}
\usepackage{graphicx}
\usepackage{amsfonts}
\usepackage[left=3.8cm,right=3.8cm,bottom=3.8cm,top=4cm]{geometry}

\DeclareGraphicsExtensions{.eps,.pdf}

\newcommand{\R}{{\mathbb R}}

\newcommand{\N}{{\mathbb N}}

\newcommand{\ka}{{\kappa}}

\newtheorem{thm}{Theorem}

\newtheorem{lemma}{Lemma}
\newtheorem{pro}{Proposition}

\theoremstyle{definition}
\newtheorem{definition}{Definition}
\newtheorem{example}{Example}
\newtheorem{remark}{Remark}

\newcommand{\Hmm}[1]{\leavevmode{\marginpar{\tiny%
$\hbox to 0mm{\hspace*{-0.5mm}$\leftarrow$\hss}%
\vcenter{\vrule depth 0.1mm height 0.1mm width \the\marginparwidth}%
\hbox to 0mm{\hss$\rightarrow$\hspace*{-0.5mm}}$\\\relax\raggedright
#1}}}

\begin{document}

\title[Symmetry of the spectrum of signed graphs]{On the symmetry of the Laplacian spectra of signed graphs}
\author{Fatihcan M. Atay}
\address{Max Planck Institute for Mathematics in the Sciences, Leipzig 04103, Germany}
\email{fatay@mis.mpg.de}
\author{Bobo Hua}
\address{School of Mathematical Sciences, LMNS, Fudan University, Shanghai 200433, China; Shanghai Center for Mathematical Sciences, Fudan University, Shanghai 200433, China}
\email{bobohua@fudan.edu.cn}

\begin{abstract}
We study the symmetry properties of the spectra of normalized Laplacians on signed graphs. We find a new machinery that generates symmetric spectra for signed graphs, which includes bipartiteness of unsigned graphs as a special case. Moreover, we prove a fundamental connection between the symmetry of the spectrum and the existence of damped two-periodic solutions for the discrete-time heat equation on the graph. \medskip \\
\textbf{MSC2010:} 05C50, 05C22, 39A12
\end{abstract}
\maketitle

\section{Introduction}


A signed graph refers to a graph together with a labeling of edges by a sign $\pm 1$, so as to represent two different types of relationships between adjacent vertices. Such graphs originate from studies of social networks where one distinguishes between friend and foe type relations \cite{Harary53}. They are also appropriate models of electrical circuits with negative resistance or neuronal networks with inhibitory or excitatory connections.

Let $G(V,E,\mu)$ be a finite, undirected, weighted graph without isolated vertices, having the vertex set $V$, edge set $E$, and edge weight $\mu:E\to \mathbb{R}_+$.
Two vertices $x,y$ are called neighbours, denoted $x\sim y$, if $(x,y)\in E$.
The weighted degree measure
$m: V\to (0,\infty)$ on the vertices is defined by $m(x)=\sum_{y:y\sim x}\mu_{xy}$.
It reduces to the combinatorial degree for the unweighted case, i.e. when $\mu=1_E,$ the constant function $1$ on $E.$
A function $\eta: E \to \{1,-1\}$  is called a sign function on $E.$ We call $\Gamma=(G,\eta)$ a \emph{signed graph}
based on $G$ with the sign function $\eta.$
Let $\ell^2(V,m)$ denote the space of real functions on $V$ equipped with the $\ell^2$ norm with respect to the measure $m.$ The \emph{normalized signed Laplacian}  $\Delta: \ell^2(V,m) \to \ell^2(V,m)$ for a signed graph $\Gamma$ is defined by
\begin{equation}
	\Delta f(x)=\frac{1}{m(x)}\sum_{y\in V:y\sim x}\mu_{xy}(f(x)-\eta_{xy}f(y)), \quad \forall f\in \ell^2(V,m).
\end{equation}
One can show that $\Delta$ is a self-adjoint linear operator on $\ell^2(V,m).$ The eigenvalues of $\Delta$ will be referred to as the spectrum of the signed graph $\Gamma$ and denoted by $\sigma(\Gamma)$.
The spectrum is always contained in the interval $[0,2]$.
The spectral properties of signed graphs have been studied by several authors, e.g.~\cite{HouLi03,Hou05,LiLi08,LiHonghaiLi09,GerminaHameedZaslavsky11,AtayTuncel,Belardo14}.


We say that the spectrum of a signed graph $\Gamma$ is \emph{symmetric} (w.r.t.
the point $\{1\}$) if $$\sigma(\Gamma)=2-\sigma(\Gamma),$$
that is, $\lambda$ is an eigenvalue of the signed Laplacian $\Delta$ of $\Gamma$ if and only if $2-\lambda$ is.
For unsigned graphs, i.e.~for $\eta=1_E$, it is well known  that  the spectrum is symmetric if and only if the graph is bipartite \cite{Chung97}.
For signed graphs, however, the symmetry of the spectrum and  bipartiteness are not equivalent: While bipartite signed graphs do have symmetric spectra (see Lemma \ref{l:bipartite spectrum}), we will see that the reverse implication does not hold,
and there exist non-bipartite signed graphs which nevertheless have symmetric spectra.
In fact,
we will present a general machinery to produce symmetric spectra for signed graphs that has no counterpart in the unsigned case. This can be seen as one of the major differences between spectral theories of unsigned and signed graphs.

To state our result, we use the concept of a \emph{switching function}, i.e., a function $\theta: V\to \{1,-1\}$:
Given $\Gamma=(G,\eta)$, a switching function $\theta$ can be used to define a new
graph $\Gamma^{\theta}=(G,\eta^\theta)$ having the sign function $\eta^{\theta}(xy)=\theta(x)\eta_{xy}\theta(y)$.
For a signed graph $\Gamma=(G,\eta),$ we denote by $-\Gamma=(G,-\eta)$ the signed graph with the opposite sign function.
We say that two signed weighted graphs
$\Gamma=(V,E,\mu,\eta)$ and $\Gamma'=(V',E',\mu',\eta')$ are \emph{isomorphic}, denoted
$\Gamma\simeq\Gamma'$, if there is a bijective map  $S:V\to V'$ such that $Sx\sim Sy$ iff $x\sim y,$ and $\mu_{SxSy}'=\mu_{xy}$ and $\eta_{SxSy}'=\eta_{xy}$ for any $x\sim y$.
That is, two signed graphs are isomorphic if they only differ up to a relabeling of vertices.
Based on these notions we prove the following in Theorem \ref{t:symmetry of the spectrum}:

\emph{Given $\Gamma=(G,\eta)$, if there is a
switching function $\theta: V\to \{1,-1\}$ such that
$\Gamma^{\theta}\simeq -\Gamma$, then the spectrum of $\Gamma$ is symmetric.}



Motivated by this new machinery, we provide a nontrivial example (Example \ref{ex:symmetric spectrum}) of a signed graph that possesses a symmetric spectrum but is not bipartite.

Although the symmetry of the spectrum of a signed graph is a purely algebraic property, we find that it plays an important role in the analysis and dynamics on the graph. We say that a function $f:(\N\cup\{0\})\times V\to \R$ solves the
discrete-time heat equation on $(G,\eta)$ with the initial data
$g:V\to \R$, if for any $n\in \mathbb{N}\cup\{0\}$ and $x\in V$,
\begin{equation}\label{e:discrete time heat}
\left\{\begin{array}{ll}f(n+1,x)-f(n,x)=-\Delta f(n,x), & \\
f(0,x)=g(x). &
\end{array}
\right.
\end{equation}
This definition mimics continuous-time heat equations on Euclidean domains or Riemannian manifolds.
Among all solutions to the discrete-time heat equation, a special class, namely damped $2$-periodic
solutions (see Definition \ref{d:periodic solution}),
will be of particular interest. These solutions are motivated from the well-known oscillatory solutions on unsigned graphs:
For a bipartite unsigned graph $G(V,E,\mu)$ with bipartition $V=V_1\cup V_2,$ the solution to the discrete-time heat equation with initial data $f(0,\cdot)=1_{V_1}$ (i.e. the characteristic function on $V_1$) is given by (see e.g.
\cite[pp. 43]{Grigoryan09})
\begin{equation} \label{e:periodic_solution}
f(n,\cdot)=\left\{\begin{array}{ll}1_{V_1}, & n \text{ even}, \\
1_{V_2}, & n  \text{ odd}.
\end{array}
\right.
\end{equation}
That is, the solution oscillates between two phases, $1_{V_1}$ and $1_{V_2},$ which justifies the name of $2$-periodic solution. We generalize such solutions to signed graphs and also allow temporal damping with decay rate $\lambda\in[0,1]$. (In the example of \eqref{e:periodic_solution}, $\lambda=1$.)
%
We show that this analytic property of the solutions, i.e. the periodicity of order two, is deeply connected to the symmetry of the spectrum of the signed graph.
Precisely, we prove the following in Theorem \ref{t:symmetry and periodic}:

\textit{
Let $(G,\eta)$ be a signed  graph. Then $u$ is a damped $2$-periodic solution with decay rate $\lambda$ ($\lambda\neq0$) if and only if
 $$u=f+g,$$
where $f$ and $g$ are eigenfunctions corresponding to the eigenvalues $1-\lambda$ and $1+\lambda$, respectively, of the normalized Laplacian.}


In the last section, we study the spectral properties of the motif
replication of a signed graph. A subset of vertices, say $\Omega$,
of a signed graph $\Gamma=(G,\eta)$ is sometimes referred to as a motif. By motif
replication we refer to the enlarged graph $\Gamma^{\Omega}$ that
contains a replica of the subset $\Omega$ with all its
connections and weights; see \cite{AtayTuncel} or Section~\ref{s:motif}. Let $\Delta_{\Omega}$ be the signed Laplacian on
$\Omega$ with Dirichlet boundary condition whose spectrum is denoted by $\sigma(\Delta_{\Omega})$; see \cite{BHJ12} for the unsigned case or
Section~\ref{s:motif}.
We prove in Theorem \ref{t:motif replication} that
$\sigma(\Delta_{\Omega})\subset \sigma(\Gamma^{\Omega}).$
This follows from a
discussion with Bauer-Keller \cite{BauerKeller12} which obviously
generalizes \cite[Theorem~13]{AtayTuncel}.
As a consequence, if the subgraph $\Omega$ admits damped 2-periodic solutions, then so does the larger graph $\Gamma^\Omega$ after replication.



The paper is organized as follows. In the next section we introduce
the concepts of signed graphs and normalized signed Laplacians, and
study their spectral properties. In Section~\ref{s:symmetry}, we explore a general machinery to create symmetry in the spectrum.
Section~\ref{s:Periodic solutions} is devoted to
damped $2$-periodic solutions of the discrete-time heat equation and their connection to the
symmetry of the spectrum. The last
section contains the spectral properties of motif replication.


\section{Basic properties of signed graphs}\label{s:basic properties signed graphs}
In this section, we study the basic properties of signed graphs and
the spectral properties of the normalized signed Laplacian.
Let $G(V,E)$ be a finite (combinatorial) graph with the set of vertices $V$ and the set of edges $E$ where $E$ is a symmetric subset of $V\times V$.
A graph is called connected if for any $x,y\in V$ there is a finite sequence of vertices, $\{x_i\}_{i=0}^n,$ such that $$x=x_0\sim x_1\sim \cdots\sim x_n=y.$$ In this paper, we consider finite, connected, undirected graphs without isolated vertices.

We assign symmetric weights on edges, $$\mu:E\to (0,\infty),\quad E\ni (x,y)\mapsto \mu_{xy}$$ which satisfies $\mu_{xy}=\mu_{yx}$ for any $x\sim y,$  and call the triple $G(V,E,\mu)$ a \emph{weighted graph}.
The special case of $\mu=1_E$ is also referred to as an unweighted graph.  For any $x\in V$, the weighted degree of $x$ is defined as  $$m(x)=\sum_{y\sim x} \mu_{xy}.$$ The weighted degree function $m: V\to (0,\infty),$ $x\mapsto m(x)$ can be
understood as a measure on $V.$ We denote by $\ell^2(V,m)$ the space of real functions
on $V$ equipped with an inner product with respect to the measure $m$, defined by $\langle u,v\rangle=\sum_{x\in
V}u(x)v(x)m(x)$ for $u,v\in \ell^2(V,m).$


\begin{definition}[Weighted signed graphs]\label{d:signed graphs} Let $G(V,E,\mu)$ be a
weighted graph. A symmetric function $\eta: E\to \{1,-1\},$
$E\ni(x,y)\mapsto\eta_{xy}$, is called a \emph{sign function} on $G$.
We refer to the quadruple $(V,E,\mu,\eta)=(G,\eta)$ as a \emph{weighted signed
graph}.
\end{definition}

In the special case $\eta=1_E$, $(G,\eta)$ is called an unsigned graph. In the following, by signed graphs we always
mean weighted signed graphs. For convenience, we define the \emph{signed weight} of a
signed graph $(V,E,\mu,\eta)$ by $\ka=\eta \mu: E\to \R,$ i.e. $\ka_{xy}=\eta_{xy}\mu_{xy}$ for any $x\sim y.$


%

Let $\Gamma=(V,E,\mu,\eta)$ be a signed graph. The \emph{normalized signed Laplacian} of $\Gamma$, denoted by
$\Delta_{\Gamma},$ is defined as
$$\Delta_{\Gamma} f(x)=\frac{1}{m(x)}\sum_{y\sim x}\mu_{xy}(f(x)-\eta_{xy}f(y)),\ \ \ \forall\ f:V\to\R.$$


The adjacency matrix of a signed graph $\Gamma$ is defined
by \[A_{\Gamma}(x,y)=\left\{\begin{array}{lr} \ka_{xy}=\eta_{xy}\mu_{xy},& \mathrm{if}\ x\sim y \\
0,&\mathrm{otherwise}\end{array}\right.\] The degree matrix is defined as
$D_{\Gamma}(x,y):=m(x)\delta_{xy}$, where $\delta_{xy}=1$ if $y=x$, and $0$ otherwise.
Hence, the normalized Laplacian
of the signed graph $\Gamma=(G,\eta)$ can be expressed as the matrix
$$\Delta_{\Gamma}=I-D_{\Gamma}^{-1}A_{\Gamma},$$
or as an operator on  $\ell^2(V,m)$,
$$\Delta_{\Gamma} f(x)=f(x)-\frac{1}{m(x)}\sum_{y\sim x}f(y)\ka_{xy}.$$
for any $f:V\to\R$.
We call $P_{\Gamma}=D_{\Gamma}^{-1}A_{\Gamma}$ the \emph{generalized transition matrix}
in analogy to the transition matrix of the simple random walk on an
unsigned graph. One notices that the row sum of $P_{\Gamma}$ is not necessarily equal to $1$ for a generic signed graph, which is an obvious difference between signed and unsigned graphs. We will often omit the subscripts
and simply write $\Delta$ and $P$ for
$\Delta_{\Gamma}$ and $P_\Gamma$, respectively unless we want to emphasize the underlying signed graph $\Gamma$.

One can show that
$\Delta:\ell^2(V,m)\to
\ell^2(V,m)$ is a bounded self-adjoint linear operator on a finite dimensional Hilbert space, and hence the spectrum is real and discrete. We denote by $\sigma(T)$ the spectrum of a linear operator $T$ on
$\ell^2(V,m)$. By the spectrum of a signed graph $\Gamma$, denoted by $\sigma(\Gamma),$ we mean the spectrum of the normalized signed Laplacian of $\Gamma,$ $\sigma(\Delta_{\Gamma}).$
Since $\Delta=I-P,$
we have $\sigma(\Delta)=1-\sigma(P)$, where the right
hand side is understood as $\{1-\lambda:\lambda\in
\sigma(P)\}.$ By the Cauchy-Schwarz inequality, one can show that the
operator norm of $P$ on $\ell^2(V,m)$ is bounded by 1.
Hence, $\sigma(P)\subset [-1,1]$, and consequently $\sigma(\Gamma)\subset [0,2]$ for any
signed graph $\Gamma.$

Given a weighted graph $G(V,E,\mu)$, we let
$$\mathcal{G}=\{(G,\eta)\, | \, \eta \ {\rm is\ a\ sign\ function}\}$$
denote the set of all signed graphs with a common underlying weighted graph
$G$.
We distinguish two special cases where the edges have all positive or all negative signs, namely $(G,+):=(G,1_E)$ and $(G,-):=(G,-1_E)$, respectively.

Given a signed graph $\Gamma=(G,\eta)$, a function $\theta: V\to \{1,-1\}$ is called a \emph{switching function} (on $V$). Using the switching function $\theta,$ one can define a new signed
graph as \begin{equation}\label{d:switched}\Gamma^{\theta}=(G,\eta^\theta),\quad\quad \mathrm{where}\quad
\eta^{\theta}(xy)=\theta(x)\eta_{xy}\theta(y),\quad x\sim y.\end{equation}
For a switching function $\theta,$ let $S^{\theta}$
denote the diagonal matrix defined as
$S^{\theta}(x,y):=\theta(x)\delta_{xy}$.
Then the adjacency
matrix of $\Gamma^{\theta}$ can be written as
$$A_{\Gamma^{\theta}}=S^{\theta}A_{\Gamma} S^{\theta},$$
Note that the degree matrix
is invariant under the switching operation, i.e.
$D_{\Gamma^{\theta}}=D_{\Gamma},$ and
$(S^{\theta})^{-1}=S^{\theta}$.
Clearly, $S^{\theta}D=DS^{\theta}$,
and hence
$\Delta_{\Gamma^{\theta}}=(S^{\theta})^{-1}\Delta_{\Gamma}
S^{\theta}$, which implies that the spectrum is invariant under the
switching operation $\theta.$ This observation yields the following
lemma.
\begin{lemma}\label{l:invariant by switching}
Let $(G,\eta)$ be a signed graph and $\theta:V\to\{1,-1\}$ a
switching function. Then the switched signed graph $\Gamma^{\theta}$ defined in \eqref{d:switched}
 has the same spectrum as $\Gamma$, i.e.
$$\sigma(\Gamma^{\theta})=\sigma(\Gamma).$$
Moreover, if $f:V\to\R$ is an eigenfunction of $\Delta_\Gamma$ corresponding to the eigenvalue
$\lambda$, then the function $f^{\theta}:V\to\R$, defined by
$f^{\theta}(x)=\theta(x)f(x)$ for
$x\in V$, is an
eigenfunction of $\Delta_{\Gamma^{\theta}}$ corresponding to the eigenvalue $\lambda$.
\end{lemma}

Given a weighted graph $G(V,E,\mu)$ with a fixed labeling of
vertices, we introduce an equivalence relation on the set of all
signed graphs $\mathcal{G}$ based on $G$:
Two signed graphs
$\Gamma_1,\Gamma_2\in \mathcal{G}$ are called equivalent, denoted
$\Gamma_1\sim\Gamma_2,$ if there exists a switching function
$\theta: V\to \{1,-1\}$ such that $\Gamma_1^{\theta}=\Gamma_2$.
For $\Gamma\in \mathcal{G}$, we denote by $\overline{\Gamma}$ the
equivalence class of $\Gamma$ and by
$\overline{\mathcal{G}}:=\{\overline{\Gamma}:\Gamma\in\mathcal{G}\}$
the set of all equivalence classes.

For a finite signed graph $\Gamma=(G,\eta)$, we order the eigenvalues of the
normalized Laplacian in a nondecreasing way:
$$0\leq \lambda_1\leq \lambda_2\leq\cdots\leq \lambda_N\leq 2,$$
where $N=|V|.$ We always denote the smallest and largest eigenvalues by
$\lambda_1(\Gamma)$ and $\lambda_N(\Gamma)$, respectively.
The following result is well known for
unsigned graphs (e.g.~\cite{Chung97}).

\begin{lemma}\label{l:spectrum unsigned}
Let $G(V,E,\mu)$ be a finite connected unsigned graph. Then the following
are equivalent:
\begin{enumerate}[(i)]
	\item $G$ is bipartite.
	\item $\lambda_N(G)=2.$
	\item The spectrum of $G$ is symmetric with respect to 1, i.e.
	$\sigma(G)=2-\sigma(G)$.
\end{enumerate}
\end{lemma}

\noindent Clearly, statement $(iii)$ is equivalent to saying that the
spectrum of the transition matrix $P=D^{-1}A$ is symmetric with respect to 0, i.e.~$\sigma(P)=-\sigma(P).$

In this section, we characterize property $(ii)$ of Lemma \ref{l:spectrum unsigned} for signed graphs. Property $(iii),$ i.e. the symmetry of the spectrum, will be postponed to the next section.

Let $C$
be a cycle, i.e. $C=\{x_i\}_{i=0}^k$ such that $x_0\sim x_1\sim\cdots\sim x_k\sim x_0,$ in a signed graph $\Gamma=(G,\eta)$.
The sign of $C$ is defined as $${\rm sign}(C)=\prod_{e\in C}\eta_e$$ where
the product is taken over all edges $e$ in the cycle.
A signed graph $\Gamma=(G,\eta)$ is called
\emph{balanced} if every cycle in
$\Gamma$  has positive sign.
The following characterization of balanced signed graphs is well
known \cite[Theorem~1]{Harary53}, \cite{Zaslavsky82}, \cite[Corollary~2.4]{Hou05} and \cite[Theorem~1]{LiHonghaiLi09}.

\begin{lemma}
Let $\Gamma=(G,\eta)$ be a signed graph. Then the following
statements are equivalent:
\begin{enumerate}[(a)]
\item $\Gamma$ is balanced.
\item $\Gamma\in \overline{(G,+)}.$
\item There exists a partition of $V,$ $V=V_1\cup V_2$ and $V_1\cap
V_2=\emptyset$, such that every edge connecting $V_1$ and $V_2$ has negative sign and every edge within $V_1$ or $V_2$ has positive sign.
\item $\lambda_1(\Gamma)=0.$
\end{enumerate}
\end{lemma}

For a signed graph $\Gamma=(G,\eta)$, the
\emph{reverse signed graph} is defined as $-\Gamma:=(G,-\eta)$.
Clearly $(G,-)=-(G,+).$
A signed graph $\Gamma=(G,\eta)$ is called
\emph{antibalanced} if $-\Gamma$ is balanced.
Furthermore, since $P_{-\Gamma}=-P_{\Gamma},$ we have
\begin{equation}\label{e:minusoperation}
	\sigma(-\Gamma)=2-\sigma(\Gamma).
\end{equation}
Based on these observations, one obtains the following lemma; see also \cite[Theorem~1]{LiHonghaiLi09}.

\begin{lemma}
Let $\Gamma=(G,\eta)$ be a signed graph. Then the following are
equivalent:
\begin{enumerate}[(a)]
\item $\Gamma$ is antibalanced.
\item $\Gamma\in \overline{(G,-)}.$
\item There exists a partition of $V,$ $V=V_1\cup V_2$ ($V_1\cap
V_2=\emptyset$), such that every edge connecting $V_1$ and $V_2$ has
positive sign and each edge within $V_1$ or $V_2$ has negative
sign.

\item $\lambda_N(\Gamma)=2.$
\end{enumerate}
\end{lemma}

\begin{remark} For any weighted graph,
$$\lambda_1(G,\eta)\geq 0=\lambda_1(G,+),$$
$$\lambda_N(G,\eta)\leq 2=\lambda_N(G,-),$$
where the equalities hold only for balanced or antibalanced graphs, respectively.
\end{remark}


In the remainder of this section, we discuss the first eigenvalue and eigenvectors of signed graphs.
As usual, the smallest eigenvalue of a finite signed graph $\Gamma$ is characterized by the Rayleigh quotient
$$\lambda_1(\Gamma)=\inf_{f\neq 0}\frac12\frac{\sum_{x,y\in V}\mu_{xy}(f(x)-\eta_{xy}f(y))^2}{\sum_{x\in V}f^2(x)\mu(x)}.$$
Is there any special property for the first eigenvalue and eigenvector?
It is well known that the first eigenvalue of an unsigned weighted graph
$G$ is simple if the graph is connected. However, this is not
the case for signed graphs:

\begin{example}
Let $(K_N,+)$ be an unsigned complete graph of $N$ vertices and $\Gamma=(K_N,-).$ Then
$$\sigma(\Gamma)=\left\{\frac{N-2}{N-1},\dots,\frac{N-2}{N-1},2\right\},$$
where the multiplicity of the first eigenvalue is
$N-1$.
\end{example}

Note that the multiplicity of the first eigenvalue of a signed graph can be quite large.
The example above concerns antibalanced graphs.
We next give an example of a generic signed graph, neither balanced nor
antibalanced, whose first eigenvalue has multiplicity larger
than 1.

\begin{example}
Let $\Gamma=(C_4,\eta)$ be a cycle graph of order $4$ with edge
signs $\{1,1,1,-1\}$. Clearly, it is neither balanced nor
antibalanced. An explicit calculation shows that
$$\sigma(\Gamma)=\left\{1-\frac{\sqrt2}{2},1-\frac{\sqrt2}{2},1+\frac{\sqrt2}{2},1+\frac{\sqrt2}{2}\right\},$$
where the multiplicity of the first eigenvalue is $2$.
Thus, the first eigenvalue of a signed graph may have high multiplicity.
\end{example}

A well-known fact is that the first eigenfunction of an unsigned graph can be chosen to be positive everywhere (in fact, it is a constant function).
One then has the following natural question:
If $(G,\eta)$ is a signed graph with a positive first
eigenfunction, should this graph be the unsigned weighted graph
$(G,\eta)=(G,+)$?
We provide a negative answer.

\begin{thm}
For any $\overline{\Gamma}\in \overline{\mathcal{G}},$ there exists
a signed graph $\Gamma'\in\overline{\Gamma}$ such that the first
eigenvector of $\Delta_{\Gamma'}$ is nonnegative everywhere.
Moreover, if one of the first eigenvectors of $\Gamma$ vanishes
nowhere, then there exists a signed graph
$\Gamma'\in\overline{\Gamma}$ such that the first eigenvector of
$\Delta_{\Gamma'}$ is strictly positive everywhere.
\end{thm}

\begin{proof}
Let $f$ be a first eigenvector of $\Delta(\Gamma).$ We define a
switching function $\theta:V\to\R$ by
\[\theta(x)=\left\{\begin{array}{rl}-1, & f(x)<0,\\
1, & \rm{otherwise}.
\end{array}
\right.\]
Clearly, $f^{\theta}:=\theta\cdot f$ is nonnegative everywhere.
Then by Lemma \ref{l:invariant by switching}, $f^{\theta}$ is the
first eigenvector to $\Delta_{\Gamma^{\theta}}$.
The signed graph $\Gamma ' = \Gamma^{\theta}$ satisfies the assertions of the theorem.
\end{proof}

\section{Symmetry of the spectrum}\label{s:symmetry}
In this section, we study the symmetry of the spectra of signed graphs. Recall that for unsigned graphs bipartiteness is the only reason for the symmetry of the spectra. However, for signed graphs some new phenomena emerge.


A signed graph is called
\emph{bipartite} if its underlying graph is bipartite, i.e. there is a partition of $V,$ $V=V_1\cup V_2$, such that any edge in $E$ connects a vertex in $V_1$ to a vertex in $V_2.$ One
notices that the sign function plays no role in the definition of
bipartiteness. By the same techniques as in the unsigned case, one can
prove that the spectrum of a signed graph is symmetric if it is bipartite; see \cite[Lemma~4]{AtayTuncel}.

\begin{lemma}\label{l:bipartite spectrum}
If $\Gamma=(G,\eta)$ is a bipartite signed graph, then the spectrum of $\Gamma$ is symmetric.
\end{lemma}

The next proposition gives a characterization of bipartite signed
graphs.

\begin{pro}\label{p:bipartite}
A signed graph $\Gamma=(G,\eta)$ is bipartite if and only if
$\overline{\Gamma}=\overline{-\Gamma}.$
\end{pro}

\begin{proof}
$\Longrightarrow$: Let $\Gamma=(G,\eta)$ be a bipartite signed graph
with bipartition $V_1,V_2$,
i.e.~$V=V_1\cup V_2,$
$V_1\cap V_2=\emptyset,$ $V_1,V_2\neq\emptyset$ and there is no edge in the induced subgraphs $V_i,$ $i=1,2$. Set
 \[
\theta(x)=\left\{\begin{array}{rl}1, & \text{if }x\in V_1,\\
-1, & \text{if }x\in V_2.\end{array} \right.
\]
Then we have $\Gamma^{\theta}=-\Gamma.$

$\Longleftarrow$: By $\overline{\Gamma}=\overline{-\Gamma},$ there
exists a switching function $\theta:V\to\{1,-1\}$ such that
\begin{equation}\label{e:eq1}\Gamma^{\theta}=-\Gamma.\end{equation}Let $V_1=\{x\in V \, | \, \theta(x)=1\}$ and $V_2=\{x\in V \, | \, \theta(x)=-1\}.$
Then there are no edges within the induced subgraphs $V_1$ and $V_2$.
Indeed, if there were an edge in the subgraph $V_1$ or $V_2$, this would contradict \eqref{e:eq1}. Hence we obtain a bipartition, $V_1\cup V_2$, of $G.$
\end{proof}

Recall that for unsigned weighted graphs the symmetry of the spectrum is completely equivalent to the bipartiteness
of the graph (see  in Lemma \ref{l:spectrum unsigned}$(iii)$).
A main difference in signed graphs is
that there are more structural conditions which may create symmetric spectra.
In the following, we present a general machinery to produce symmetric spectrum for signed
graphs that has no counterpart in the unsigned case.

As defined in the introduction, two signed graphs $(G,\eta),(G',\eta')$ are called isomorphic if they have same combinatorial, weighted, and signed graph structure. Hence
$$\sigma(G,\eta)=\sigma(G',\eta').$$
Now we are ready to prove one of the main results of this paper.


\begin{thm}\label{t:symmetry of the spectrum}
Let $\Gamma=(G,\eta)$ be a signed weighted graph. If there is a
switching function $\theta: V\to \{1,-1\}$ such that
$\Gamma^{\theta}\simeq -\Gamma$, then the spectrum of $\Gamma$ is symmetric.
\end{thm}

\begin{proof}[Proof of Theorem \ref{t:symmetry of the spectrum}]
Combining Lemma \ref{l:invariant by switching},
\eqref{e:minusoperation} and the invariance of the spectrum under
the isomorphism, we have
$$\sigma(\Gamma)=\sigma(\Gamma^{\theta})=\sigma(-\Gamma)=2-\sigma(\Gamma).$$
This proves the theorem.
\end{proof}

Note that for bipartite signed graphs  $\overline{\Gamma}=\overline{-\Gamma}$ by Proposition \ref{p:bipartite}. This
indicates that Lemma \ref{l:bipartite spectrum} is a special case of
Theorem \ref{t:symmetry of the spectrum}. In the following, inspired by Theorem \ref{t:symmetry of the spectrum}, we provide an example of a signed graph with symmetric spectrum, although it is non-bipartite.

\begin{example}\label{ex:symmetric spectrum}
Let $\Gamma$ be the signed graph shown in Figure 1, with $+1$ and $-1$ edge weights as indicated.
Consider the switching function $\theta$ given by
\[\theta(i)=\left\{\begin{array}{rl}-1, & i=2,4,\\
1, & \rm{otherwise}. \end{array}
\right.\]
Now the permutation $T=(1,3)\in S_5$, which is an isomorphism of
signed graphs, transforms $\Gamma^{\theta}$ to $-\Gamma.$ By Theorem
\ref{t:symmetry of the spectrum}, the spectrum of $\Gamma$ is
symmetric with respect to $1$.
Indeed, an explicit calculation gives $$\sigma(\Gamma)=\left\{1\pm \frac{\sqrt{4-\sqrt3}}{3},1\pm
\frac{\sqrt{4+\sqrt3}}{3},1\right\}.$$
Thus, $\Gamma$ is a nontrivial (i.e.~non-bipartite) example for the symmetry of the spectrum.
\end{example}

\begin{figure}
\centering
\includegraphics[width=0.8\textwidth]{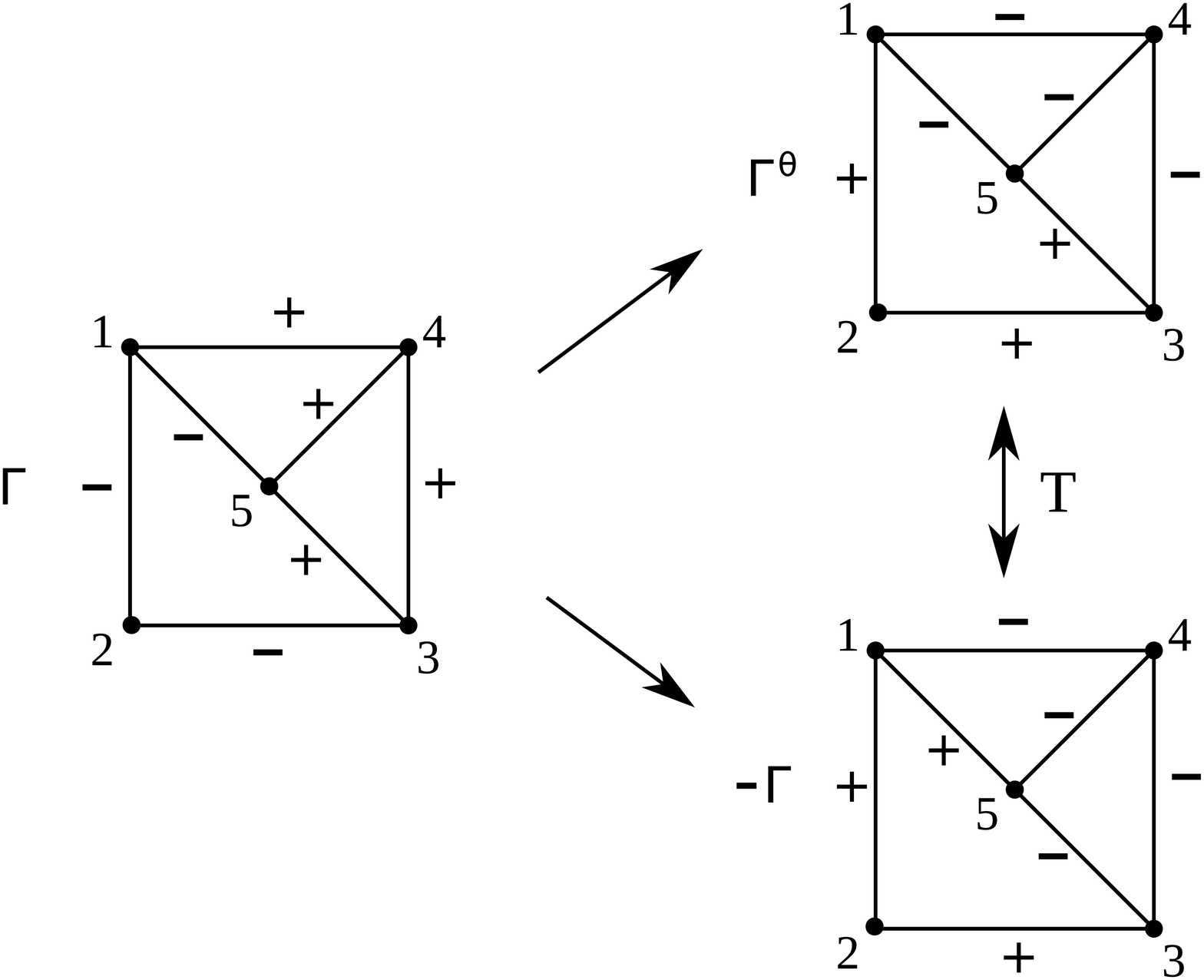}
\caption{A non-bipartite signed graph with symmetric spectrum.
The eigenvalues are $1\pm \frac{\sqrt{4-\sqrt3}}{3},1\pm
\frac{\sqrt{4+\sqrt3}}{3}$, and 1.} \label{fig1}
\end{figure}







\section{Damped two-periodic solutions}
\label{s:Periodic solutions}
In this section, we connect the symmetry of the spectrum of signed graphs to the existence of period-two oscillatory solutions for the discrete-time heat equation on the graph.

Let
$(G,\eta)$ be a finite signed graph. We say that a function
$f:(\N\cup{0})\times V\to \R$ satisfies the discrete-time heat
equation on $(G,\eta)$ if
\begin{equation*}
\left\{\begin{array}{ll}f(n+1,x)-f(n,x)=-\Delta f(n,x),
\quad n\in \mathbb{N}\cup \{0\},& \\
f(0,x)=g(x). &
\end{array}
\right.
\end{equation*}
For simplicity, we denote the function $f_n:V\to\R,$
$n\geq 0$, by $f_n(x):=f(n,x)$. Then the
heat equation can be written as
\begin{equation} \label{heateq}
\left\{\begin{array}{ll}f_{n+1}=P f_n, & \\
f_0=g, &
\end{array}
\right.
\end{equation}
where $P=D^{-1}A$ is the generalized transition matrix.
As usual, the notation $\|\cdot\|$ and $\langle \cdot,\cdot\rangle$ will denote the $\ell^2$ norm and inner product, respectively,  of the Hilbert space $\ell^2(V,m).$

\begin{definition}\label{d:periodic solution}
A function $u\neq 0$ is called a generic damped periodic solution of order
two to the discrete-time heat equation \eqref{heateq}, or a \emph{damped $2$-periodic solution} for short, if $u$ is an eigenfunction of $P^2$ but not of $P$; that is, if
there exists a constant $\lambda\geq0$ such that $P^2u=\lambda^2 u$
and $Pu\neq \pm\frac{\|Pu\|}{\|u\|}u.$ We call $\lambda$ the \emph{decay rate}.
\end{definition}

\begin{pro} Let $u$ be a damped $2$-periodic solution of \eqref{heateq} with decay rate $\lambda$. Then the following hold:
  \begin{enumerate}[(a)]
  \item $ku$ is a damped $2$-periodic solution for any $k\neq0.$
  \item $\lambda=\sqrt{\frac{\|P^2u\|}{\|u\|}}=\frac{\|Pu\|}{\|u\|}$.
  \item $0<\lambda\le 1.$
  \end{enumerate}
\end{pro}

\begin{proof}
Statement $(a)$ follows by definition.
By $P^2u=\lambda^2u$ and $u\neq 0,$ we have
$\lambda=\sqrt{\frac{\|P^2u\|}{\|u\|}}.$ Moveover, noting that $P$
is a self-adjoint operator with respect to the inner product
on $\ell^2(V,m)$, we obtain
$$\|Pu\|^2=\langle Pu,Pu\rangle=\langle
P^2u,u\rangle=\lambda^2\|u\|^2.$$
This proves statement $(b)$.
For $(c)$, to prove that $\lambda\neq0$, we argue by contradiction.
Suppose $\lambda=0$. Then by statement $(b)$ we have $\|Pu\|=0$, and thus
$Pu=0$, which implies that $u$ is an eigenvector of $P$ pertaining to the eigenvalue $0$.
This, however, contradicts the definition of $2$-periodic solutions. Finally, $\lambda\le 1$ follows from the fact that all eigenvalues of $P$ belong to the interval $[-1,1]$ and $P^2 u = \lambda^2 u$ by definition.
\end{proof}

For a $2$-periodic oscillatory solution $u$ with decay rate $\lambda$,
we set $v=\frac1{\lambda}Pu$. (Notice that $v$ is not parallel to $u$ by definition).
Hence we have the system of equations
\begin{equation}\label{e:system of equations to periodic}\left\{\begin{array}{ll}Pu=\lambda v, & \\
Pv=\lambda u. &
\end{array}
\right.\end{equation}
If we set $u$ as the initial data of the
discrete-time heat equation, then the solution to \eqref{heateq}
reads
\begin{equation}\label{e:2 periodic solution}f_n=P^nu=\left\{\begin{array}{ll}\lambda^n u, & n \text{ even}, \\
\lambda^n v, & n  \text{ odd}.
\end{array}
\right.\end{equation}
Since the vectors $u$ and $v$ are linearly
independent, the solution $f_n$ oscillates between two phases, $u$ and $v,$
with an amplitude that decays exponentially at a rate $\lambda$, since $\lambda \le 1$.
This motivates the meaning of damped $2$-periodic
solution given in Definition \ref{d:periodic solution}.

We study the machinery to produce a $2$-periodic solution.
We will see that the existence of $2$-periodic solutions is
equivalent to certain symmetry of the spectrum of the normalized
signed Laplacian operator.
The next theorem states that all damped $2$-periodic solutions are in fact encoded by the symmetry of the spectrum.

\begin{thm}\label{t:symmetry and periodic} Let $(G,\eta)$ be a signed  graph. Then $u$ is a damped $2$-periodic solution with decay rate $\lambda$ if and only if
 $$u=f+g,$$
where $f$ and $g$ are eigenfunctions corresponding to the eigenvalues $1-\lambda$ and $1+\lambda$, respectively, of the normalized Laplacian $\Delta$, with $\lambda\neq0$.
\end{thm}

\begin{proof}
$\Leftarrow:$ By the assumption and the fact that $\Delta=I-P$,
$f$ and $g$ are eigenfunctions pertaining to the eigenvalues $\lambda$
and $-\lambda$, respectively, of $P$. Hence $Pu=\lambda(f-g)$ and
$P^2u=\lambda^2u.$ Since $f,g$ are eigenvectors pertaining to
different eigenvalues, $\langle f,g\rangle=0.$ It is easy to check
that $u=f+g$ and $f-g$ are linearly independent. This yields that
$u$ is not an eigenvector to $P$, and proves the assertion.

$\Rightarrow:$ Setting $v=\frac1{\lambda}Pu,$ we have the system
\eqref{e:system of equations to periodic} for $u$ and $v$.
Define $f:=\frac12(u+v)$ and $g:=\frac{1}{2}(u-v)$.
It is easy to see that $f$ and $g$ are nonzero since
$u$ and $v$ are linearly independent. By definition $u=f+g$.
Direct calculation shows that $Pf=\lambda f$ and $Pg=-\lambda g$.
This completes the proof.
\end{proof}

\begin{remark}
\begin{enumerate}[(a)] \item We have the following corollary of Theorem \ref{t:symmetry and periodic}: If there is a damped $2$-periodic solution with decay rate $\lambda$ on a signed graph $\Gamma$, then both
$1-\lambda$ and $1+\lambda$ belong to the spectrum of $\Gamma$.
\item For signed weighted graphs whose spectrum is symmetric with
respect to 1, one can construct many $2$-periodic
solutions by virtue of Theorem \ref{t:symmetry and periodic}.\end{enumerate}
\end{remark}

Finally, as a consequence of Theorem \ref{t:symmetry and periodic}, we have the following formulation.
Let $\Gamma=(G,\eta)$ be a signed weighted graph.
We set
$$\Lambda(\Gamma):=\{\lambda: \lambda\neq 1, \lambda\in
\sigma(\Gamma), 2-\lambda\in\sigma(\Gamma)\}.$$
We denote by $E_{\lambda}$ the eigenspace pertaining to the eigenvalue $\lambda$ of $\Delta_\Gamma.$
Then the set of damped $2$-periodic solutions, denoted by
$\mathcal{P}$, has the representation
$$\mathcal{P}=\bigcup_{\lambda\in\Lambda(\Gamma)}E_{\lambda}\oplus E_{2-\lambda}\setminus(E_{\lambda}\oplus 0\cup0\oplus E_{2-\lambda}).$$

\section{Motif replication}\label{s:motif}
In this section, we use the normalized Laplace operator with
Dirichlet boundary conditions to study the spectral changes under
motif replication.

The term \emph{motif} refers to a subgraph $\Omega$ of a signed graph $\Gamma$.
By motif replication we refer to the operation of appending
an additional copy of $\Omega$ together with all its connections and corresponding weights, yielding the enlarged graph denoted by
$\Gamma^{\Omega}$.
More precisely, let $\Omega$ be a subgraph on vertices $\{x_1,\dots,x_n\}$ and let
$\Omega'=\{x_1',\dots,x_n'\}$ be an exact replica of $\Omega$.
Then the enlarged graph $\Gamma^\Omega$, obtained by replicating $\Omega$, is defined on the vertex set
$V(\Gamma^{\Omega})=V(\Gamma)\cup \Omega'$ with the signed edge weights
given by
\[
\ka=\begin{cases} \ka(x,y), & x,y\in V(\Gamma), \\
\ka(x_i,x_j), & x_i',x_j'\in \Omega',\\
\ka(x_i,y), & x_i'\in \Omega', y\in V(\Gamma)\setminus \Omega,\\
\end{cases}
\]
using the edge weights $\kappa(x,y)$ from the original graph $\Gamma$
\cite{AtayTuncel}.
In particular, there are no edges between $\Omega$ and $\Omega'$ in the replicated graph $\Gamma^\Omega$.

Let $\Omega$ be a finite subset of $V$ and $\ell^2(\Omega,m)$ be
the space of real-valued functions on $\Omega$ equipped with the
$\ell^2$ inner product. Note that every function
$f\in\ell^2(\Omega,m)$ can be extended to a function
$\tilde{f}\in\ell^2(V,m)$ by setting $\tilde{f}(x)=0$ for all
$x\in V\setminus \Omega$. The Laplace operator with Dirichlet
boundary conditions $\Delta_\Omega:
\ell^2(\Omega,m) \to \ell^2(\Omega,m)$ is defined as
$$\Delta_\Omega f = (\Delta \tilde{f})_{|\Omega}.$$
Thus, for $x\in
\Omega$ the Dirichlet Laplace operator is given pointwise by
\begin{align*}
\Delta_\Omega f(x) = f(x) - \frac{1}{m(x)}\sum_{y\in\Omega: y\sim x} \ka_{xy}f(y) \\
=
\tilde{f}(x) - \frac{1}{m(x)}\sum_{y\sim x} \ka_{xy}\tilde{f}(y).
\end{align*}
A simple calculation shows that $\Delta_\Omega$ is a positive self-adjoint operator. We arrange the eigenvalues of the Dirichlet
Laplace operator $\Delta_\Omega$ in nondecreasing order,
$\lambda_1(\Omega) \leq \lambda_2(\Omega) \leq \dots \leq
\lambda_N(\Omega)$, where $N=|\Omega|$ denotes the cardinality of the set
$\Omega$.


We next prove that all eigenvalues of the Dirichlet Laplacian on
$\Omega$ are preserved in the motif replication.

\begin{thm}\label{t:motif replication} Let $\Omega$ be a motif
in a signed weighted graph $\Gamma$ and $\Gamma^{\Omega}$ be the
new graph obtained after replicating $\Omega$. Then
$\sigma(\Delta_{\Omega})\subset \sigma(\Gamma^{\Omega}).$
\end{thm}

\begin{proof}
Let $f$ be an eigenvector of the Dirichlet Laplacian
$\Delta_{\Omega}$ corresponding to eigenvalue $\lambda$. Then by direct calculation
it can be seen that the following function is an eigenvector of the Laplacian
on $\Gamma^{\Omega}$:
\begin{equation}\label{e:motif mod}f'(x)=\left\{\begin{array}{rl}f(x), & x\in \Omega, \\
-f(x), & x\in \Omega',\\
0, & \rm{otherwise},
\end{array}
\right.\end{equation} where $\Omega'$ is the copy of $\Omega.$ This proves the
theorem.
\end{proof}

\begin{remark}
By Theorem \ref{t:motif replication}, if the motif $\Omega$ supports a damped 2-periodic
solution $u$ with respect to the Dirichlet boundary condition, then $u',$ defined as in \eqref{e:motif mod}, is a damped 2-periodic solution for the replicated graph $\Gamma^{\Omega}.$
In this way, symmetric eigenvalues $\lambda$ and 2-periodic solutions with decay rate $\lambda$
are carried over to the larger graph after motif replication.
\end{remark}



\bigskip
\noindent \textbf{Acknowledgements.} The authors thank the ZiF (Center for Interdisciplinary Research) of Bielefeld University, where part of this research was conducted under the program \emph{Discrete and Continuous Models in the Theory of Networks}.
FMA acknowledges the support of the European Union's 7th Framework Programme under grant \#318723 (MatheMACS).
BH is supported by NSFC, grant no. 11401106.


\end{document}